\documentclass[a4paper,english]{article}
\usepackage[utf8]{inputenc}
\usepackage[T1]{fontenc}
\usepackage{lmodern}
\usepackage{geometry}
\usepackage{microtype}
\usepackage{amssymb}
\usepackage{mathtools}
\usepackage{amsthm}
\usepackage{comment}
\newtheorem{thm}{Theorem}
\newtheorem{lem}{Lemma}

\theoremstyle{definition}
\newtheorem{dfn}{Definition}

\theoremstyle{remark}
\newtheorem{rem}{Remark}
\usepackage{hyperref}
\usepackage[italicdiff]{physics}
\usepackage{booktabs}
\usepackage{accents}
\usepackage{tablefootnote}
\usepackage{subcaption}
\usepackage{tikz-cd}
\usetikzlibrary{babel}
\usepackage{tikzscale}
\usepackage{tikz}
\usepackage{pgfplots}
\pgfplotsset{compat=1.17}
\usepgfplotslibrary{fillbetween}
\usepackage{tkz-euclide}
\usepackage{enumitem}
\usepackage{capt-of}
\setlist[enumerate]{label=\normalfont(\alph*)}
\title{Topology of Toric Gravitational Instantons}
\author{Gustav Nilsson\thanks{\href{mailto:gustav.nilsson@aei.mpg.de}{gustav.nilsson@aei.mpg.de}}\\Max Planck Institute for Gravitational Physics (Albert Einstein Institute)\\Am Mühlenberg 1, D-14476 Potsdam, Germany}
\date{}
\begin{document}
\maketitle
\begin{abstract}
For an asymptotically locally Euclidean (ALE) or asymptotically locally flat (ALF) gravitational instanton $(M,g)$ with toric symmetry, we express the signature of $(M,g)$ directly in terms of its rod structure. Applying Hitchin--Thorpe-type inequalities for Ricci-flat ALE/ALF manifolds, we formulate, as a step toward a classification of toric ALE/ALF instantons, necessary conditions that the rod structures of such spaces must satisfy. Finally, we apply these results to the study of rod structures with three turning points.
\end{abstract}
\section{Introduction}
Ricci-flat Riemannian $4$-manifolds are the vacuum solutions to the Riemannian analog of the Einstein equations. Such manifolds occur in the study of quantum gravity, commonly under the name \emph{gravitational instantons} \cite{MR465052}. In this and many other contexts, it is common to restrict attention to complete and non-compact manifolds whose curvature decays sufficiently fast. The volume growth is then either quartic (ALE), cubic (AF/ALF), quadratic (ALG), or linear (ALH). An ALE manifold has a metric that is asymptotic to $\mathbb{R}^4/\Gamma$, where $\Gamma$ is a finite subgroup of $\mathrm{SO}(4)$ which acts freely on $S^3$. An ALF manifold, on the other hand, is asymptotic to a circle bundle over $\mathbb{R}^3$ or $\mathbb{R}^3/\mathbb{Z}_2$.

The classification problem for gravitational instantons is an interesting problem in which many conjectures have been made. For instance, Gibbons conjectured in 1979 that the only AF gravitational instantons are the Riemannian Schwarzschild and Riemannian Kerr spaces \cite{MR582633}, a conjecture which was proven wrong by Chen and Teo in the 2011 paper \cite{MR2831866}, where they presented a new AF gravitational instanton.

The examples mentioned so far are all \emph{toric}, i.e.,\ admit an effective action of the torus $T^2=\mathrm{U}(1)\times\mathrm{U}(1)$ by isometries, {and are simply connected, and we will refer to gravitational instantons with these properties as \emph{toric gravitational instantons}.} Toric gravitational instantons were studied in the paper \cite{MR2124693}. In that paper, the author introduced the formalism of so-called \emph{rod structures} (defined in detail in Section \ref{sec:rod_structure}), which are combinatorial objects associated with toric gravitational instantons, containing information about the $T^2$-action. Later, the rod structure formalism was used in \cite{MR2434746,MR2659115}.

The rod structure of a toric gravitational instanton $(M,g)$ with $n$ fixed points consists of a sequence of $2$-vectors $(v_0,\dots,v_n)$ with integer entries, satisfying \begin{equation}
\det\mqty(v_{i-1}&v_{i})=1,
\label{eq:determinant_normalization}
\end{equation} a condition which follows from the regularity of the instanton.

Much effort has been devoted to constructing solutions of the Einstein equations with toric symmetry. One example of such a method is the \emph{soliton method} \cite{MR1839019,vsrqyzhbacozoxogvmuo}, which was used to construct the Chen--Teo instanton (the AF gravitational instanton discussed in \cite{MR2831866}). Another method is given in the more recent paper \cite{uguvpxtckkesrlwjgulj}. Despite much progress, it is, in general, a hard problem to determine whether a rod structure can be realized as the rod structure of a toric gravitational instanton.

The rod structure of $M$ can be shown to determine its topology. In particular, the Euler characteristic is well-known to equal the number of fixed points, $n$. In this paper, we determine the signature in terms of the rod structure. {We do so by directly reconstructing $M$ as a manifold with a $T^2$-action from its rod structure. In terms of this explicit reconstruction, we determine a basis for $H_2(M)$,\footnote{Here $H_k(M)$ refers to the $k$:th \textbf{singular} homology group (with \textbf{integral} coefficients).}} and then compute the intersection form in terms of this basis. As a result, we obtain the following theorem.

\begin{thm}
Let $(M,g)$ be a toric gravitational instanton with rod structure $(v_0,\dots,v_n)$. Then $H_2(M)\cong\mathbb{Z}^{n-1}$, and $H_2(M)$ admits a basis\footnotemark\ in which the intersection form on $M$ is represented by the matrix \begin{equation}\mqty(d_1&1&0&\hdots&0&0&0\\1&d_2&1&\hdots&0&0&0\\0&1&d_3&\hdots&0&0&0\\\vdots&\vdots&\vdots&\ddots&\vdots&\vdots&\vdots\\0&0&0&\hdots&d_{n-3}&1&0\\0&0&0&\hdots&1&d_{n-2}&1\\0&0&0&\hdots&0&1&d_{n-1}),\end{equation} where $d_i=-\det\mqty(v_{i-1}&v_{i+1})$. In particular, the signature $\tau(M)$ is the signature of this matrix.
\label{thm:rod_structure_intersection_form}
\end{thm}
\footnotetext{{More precisely, a basis in the sense of $\mathbb{Z}$-modules.}}

Theorem \ref{thm:rod_structure_intersection_form} is then used together with Hitchin--Thorpe-type inequalities for ALE or ALF instantons to give necessary conditions on the possible rod structures. The Hitchin--Thorpe inequality for Ricci-flat ALE $4$-manifolds states that \begin{equation}2\qty(\chi(M)-\frac{1}{|\Gamma|})\geq 3|\tau(M)+\eta_S(S^3/\Gamma)|,\end{equation} and for Ricci-flat ALF $4$-manifolds asymptotic to a circle bundle over $\mathbb{R}^3$, the Hitchin--Thorpe inequality reads \begin{equation}2\chi(M)\geq 3\qty|\tau(M)-\frac{e}{3}+\operatorname{sgn}(e)|,\end{equation} where $e$ is the Euler number of the asymptotic circle bundle. Both $|\Gamma|$ and $e$ can be expressed in terms of the rod structure, and as a consequence, these inequalities and Theorem \ref{thm:rod_structure_intersection_form} lead to restrictions on the possible rod structures.

The particular case of rod structures with three turning points was discussed in \cite{MR2659115}, where the question was raised whether topological or other constraints could perhaps rule out certain rod structures. We address this question with the following result.

\begin{thm}
{Let $(M,g)$ be a toric ALE instanton with the rod structure shown in Figure \ref{fig:three_turning_points_rod_structure}. Then either $(a,b)=\pm(2,2)$, or $ab<0$ and $|a+b|\leq 6$ (as depicted\footnote{The meaning of the different colors will be explained in Section \ref{sec:three_turning_points}.} in Figure \ref{fig:three_turning_points_ALE_possibilities}).} If instead $M$ is ALF, then either $a=0$, $b=0$, or $(a,b)$ is one of the pairs in Figure \ref{fig:three_turning_points_ALF_possibilities}.
\label{thm:three_turning_points_possibilities}
\end{thm}
\begin{figure}
\centering
\includegraphics[width=\textwidth]{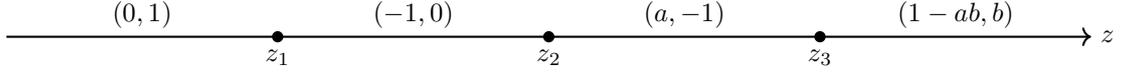}
\caption{A rod structure with three turning points.}
\label{fig:three_turning_points_rod_structure}
\end{figure}
\begin{figure}
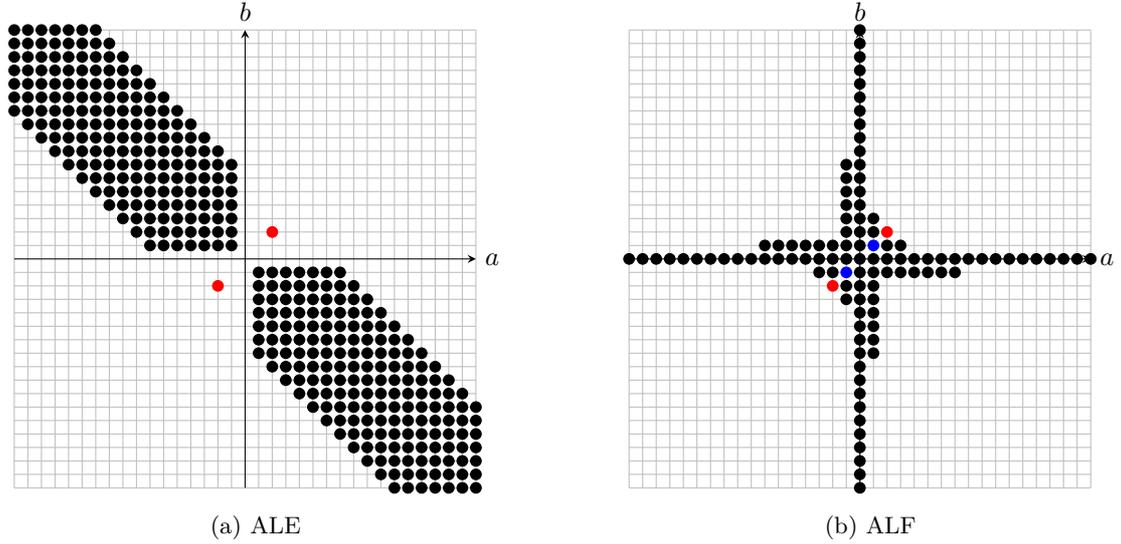

\centering
\begin{subfigure}{0.45\textwidth}
\includegraphics[width=\linewidth]{three_turning_points_ALE_possibilities.tikz}
\caption{ALE}
\label{fig:three_turning_points_ALE_possibilities}
\end{subfigure}\hfill
\begin{subfigure}{0.45\textwidth}
\includegraphics[width=\linewidth]{three_turning_points_ALF_possibilities.tikz}
\caption{ALF}
\label{fig:three_turning_points_ALF_possibilities}
\end{subfigure}
\caption{The remaining possibilities for $(a,b)$.}
\label{fig:three_turning_points_possibilities}
\end{figure}
In this paper, we have only used the assumption of Ricci-flatness minimally through the Hitchin--Thorpe inequality. Using the full set of field equations, for example as in \cite{uguvpxtckkesrlwjgulj}, one could potentially use knowledge about the topology to prove certain uniqueness results. For instance, we expect that the Riemannian black hole uniqueness conjecture in the toric case could be approached this way.
\subsection*{Overview of this paper}
In Section \ref{sec:toric_gravitational_instantons}, we give the definition of toric gravitational instanton, discuss the possible asymptotic geometries, and introduce the concept of rod structures. In Section \ref{sec:examples}, we look at some examples of known toric gravitational instantons and list their rod structures, among other properties. In Section \ref{sec:topological_conditions}, we compute the signature of a toric gravitational instanton in terms of its rod structure and then use this together with Hitchin--Thorpe-type inequalities to formulate necessary conditions on possible rod structures. Finally, in Section \ref{sec:three_turning_points}, we apply this to the case of rod structures with three fixed points and briefly discuss the case of an AF rod structure with four fixed points.
\subsection*{Acknowledgments}
The author would like to thank Mattias Dahl, Lars Andersson, Eric Ling, Steffen Aksteiner, Walter Simon, and Hans-Joachim Hein for their helpful comments and discussion. This research was supported by the IMPRS for Mathematical and Physical Aspects of Gravitation, Cosmology and Quantum Field Theory at the Max Planck Institute for Gravitational Physics.
\section{Toric Gravitational Instantons}\label{sec:toric_gravitational_instantons}
In this section, we introduce the concept of a gravitational instanton, the main concept of this paper's study. We also introduce the various possibilities for the asymptotic behavior of gravitational instantons. Then, we define toric gravitational instantons and introduce the concept of a rod structure.
\begin{dfn}\label{dfn:gravitational_instanton}
A \emph{gravitational instanton} is a Riemannian $4$-manifold $(M,g)$ which is non-compact, complete and Ricci-flat, satisfying the following.
\begin{enumerate}
\item There exists a nonincreasing function $K:[0,\infty)\to[0,\infty)$ such that \begin{equation}\int_1^\infty\frac{K(s)}{s}\dd{s}<\infty,\end{equation} along with a point $p\in M$ such that \begin{equation}\lvert\operatorname{Rm}_q\rvert_g\leq\frac{K(d(p,q))}{d(p,q)^2}\end{equation} for all $q\in M$, where $d(\cdot,\cdot)$ is the Riemannian distance function.\label{dfn:gravitational_instanton_1}
\item There exists a constant $\kappa>0$, a function $\epsilon:(0,\infty)\to(0,\infty)$ satisfying $\lim_{r\to0}\epsilon(r)=0$, along with a compact set $K\subseteq M$ such that for any $r>0$ and for any geodesic loop $\gamma$ whose length is less than $\kappa r$, and which is based at a point in $M\setminus K$, the holonomy of $\gamma$ rotates any vector by at most the angle $\epsilon(r)$.\label{dfn:gravitational_instanton_2}
\end{enumerate}
\end{dfn}
\begin{rem}
We do not require $(M,g)$ to be hyper-Kähler (see Definition \ref{dfn:hyper_kaehler}), an assumption which is sometimes included in the definition of gravitational instanton.
\end{rem}
From the conditions \ref{dfn:gravitational_instanton_1} and \ref{dfn:gravitational_instanton_2} in Definition \ref{dfn:gravitational_instanton}, it follows (see \cite{4797854091}) that $(M,g)$ has must have one out of several special geometries near infinity. The different cases are divided into ALE, ALF, ALG, and ALH depending on the volume growth, and the possible subcases \cite{4797854091} are shown in Table \ref{tab:asymptotic_geometries}. Here, we say that the \emph{boundary at infinity} of $M$ is $N$, if there exists a compact set $K\subseteq M$ such that $\overline{M\setminus K}$ is homeomorphic to $[0,\infty)\times N$, and if the homeomorphism can be chosen to map $\partial K$ onto $\{0\}\times N$. {Although the methods used in the proof of the following lemma are standard, we write down a detailed proof here since the details will be needed to obtain a slightly stronger version of the result in later sections.

\begin{lem}
The boundary at infinity is well-defined up to $h$-cobordism. In other words, if $N_1$ and $N_2$ are closed $3$-manifolds, and if they are both the boundary at infinity of $M$, then there is an $h$-cobordism between $N_1$ and $N_2$.
\label{lem:boundary_at_infinity_well_defined}
\end{lem}
\begin{proof}
For $i=1,2$, let $K_i\subseteq M$ be compact and let $\Phi_i:\overline{M\setminus K_i}\to[0,\infty)\times N_i$ be a homeomorphism such that $\Phi_i(\partial K_i)=\{0\}\times N_i$. Without loss of generality, assume that $K_1\subseteq K_2$ (otherwise, replace $K_2$ by $K_2\cup\Phi_2^{-1}([0,R]\times N_2)$, for large enough $R$).

Now pick $R_1$ such that $\overline{K_2\setminus K_1}\subseteq\Phi_1^{-1}([0,R_1])$, and then pick $R_2$ such that $\Phi_1^{-1}([0,R_1]\times N_1)\subseteq\Phi_2^{-1}([0,R_2]\times N_2)$. Then $W:=\overline{\Phi_2^{-1}([0,R_2])\setminus K_1}$ is a $4$-manifold with boundary. The boundary $\partial W$ is the disjoint union of $\Phi_1^{-1}(\{0\}\times N_1)\cong N_1$ and $\Phi_2^{-1}(\{R_2\}\times N_2)\cong N_2$, so that $W$ is a cobordism between $N_1$ and $N_2$.

Using the fact that $[0,R_2]\times N_2$ deformation retracts onto $\{0\}\times N_2$, we can use the map $\Phi_2$ to define a deformation retraction from $W$ to $\overline{K_2\setminus K_1}$, and similarly, the fact that $[0,R_1]\times N_1$ deformation retracts onto $\{0\}\times N_1$, implies that $\Phi_1^{-1}([0,R_1]\times N_1)$ deformation retracts onto $\Phi_1^{-1}(\{0\}\times N_1)$. Since $\overline{K_2\setminus K_1}\subseteq\Phi_1^{-1}([0,R_1]\times N_1)$, this implies that $W$ deformation retracts onto $\Phi_1^{-1}(\{0\}\times N_1)$. In a similar fashion, one sees that $W$ deformation retracts onto $\Phi_2^{-1}(\{R_2\}\times N_2)$, from which it follows that $W$ is a $h$-cobordism.
\end{proof}}

\begin{table}
\centering
\caption{Possible asymptotic geometries of gravitational instantons.}
\begin{tabular}{lr}
\toprule
&Boundary at infinity\\
\midrule
ALE\tablefootnote{Here, $\Gamma$ is any finite subgroup of $\mathrm{SO}(4)$ which acts freely on $S^3$.}&$S^3/\Gamma$\\
ALF-$A_{-1}$\tablefootnote{ALF-$A_{-1}$ is also called AF.}&$S^2\times S^1$\\
ALF-$A_k$,   $k\neq-1$&$L(|k+1|,-\operatorname{sgn}(k+1))$ \\
ALF-$D_2$\tablefootnote{Here, the action of $\mathbb{Z}_2$ on $S^2\times S^1$ is given by $(x,y,z,\theta)\mapsto(-x,-y,-z,-\theta)$.}&$(S^2\times S^1)/\mathbb{Z}_2$\\
ALF-$D_k$, $k\neq2$&$S^3/D_{4|k-2|}$\\
ALG&$T^2$-bundle over $S^1$\\
ALH-splitting&$\{\pm1\}\times T^3$\\
ALH-non-splitting&$T^3$\\
\bottomrule
\end{tabular}
\label{tab:asymptotic_geometries}
\end{table}
\begin{dfn}
A \emph{toric gravitational instanton} is a simply connected gravitational instanton $(M,g)$ together with an effective action of the torus $T^2=\mathrm{U}(1)\times\mathrm{U}(1)$ by isometries. We assume that the fixed point set\footnote{That is, the set of points in $M$ whose isotropy group is the entire group $T^2$.} is finite and non-empty, and that there are no points whose isotropy group is non-trivial and finite.
\end{dfn}
\begin{rem}
The requirement that there are no points with non-trivial and finite isotropy group is equivalent to requiring $M$ to be \emph{locally standard}, see \cite{ypkrsoujnujrtxzxhykz}.
\end{rem}
\subsection{Rod Structure}\label{sec:rod_structure}
\begin{lem}
Let $(M,g)$ be a toric gravitational instanton. Then the orbit space $\widehat{M}=M/T^2$ is a $2$-dimensional (smooth) manifold with corners (see \cite[p.~415]{MR2954043}). Furthermore, $\widehat{M}$ is homeomorphic to the closed upper half plane $\{z+i\rho\mid z\in\mathbb{R},\rho\geq0\}$, and under this homeomorphism, the set where $\rho>0$ corresponds to orbits whose points have trivial isotropy group. There exist points $z=z_1,\dots,z_n$ on the axis $\rho=0$, corresponding to the fixed points, dividing this axis into segments $z_i<z<z_{i+1}$ (where we take $z_0=-\infty$ and $z_{n+1}=\infty$). For each such segment, the points in the corresponding orbits all have the same isotropy group. This isotropy group is a circle subgroup of the torus, given as \begin{equation}T^2(v_i^1,v_i^2):=\{(e^{iv_i^1\theta},e^{iv_i^2\theta})\mid\theta\in\mathbb{R}\}\end{equation} for some pair $v_i=(v_i^1,v_i^2)$. {The vectors $(v_0,\dots,v_n)$ can be chosen in such a way that each vector is a pair of coprime integers, and such that \eqref{eq:determinant_normalization} holds for all $i$.}\label{lem:rod_structure_existence}
\end{lem}
\begin{proof}
Since $M$ has no points with a non-trivial finite isotropy group, the fact that $\widehat{M}$ is a $2$-manifold with corners follows from the non-compact analog of \cite[Theorem~1]{vhuuwalemetjkjihjowo}. From the discussion following the theorem, it is also clear that the interior of $\widehat{M}$ corresponds to points with trivial isotropy and that the corners correspond to fixed points. It can also be seen from this discussion that for a boundary segment, the isotropy group is constant on this segment, that it is of the form $T^2(v)$ for some pair $v$ of coprime integers, and that if $T^2(v)$ and $T^2(w)$ are the isotropy groups corresponding to two segments meeting in a corner, then $\det\mqty(v&w)=\pm1$. {It remains to show that $\widehat{M}$ is homeomorphic to the closed upper half plane, for once this is done, the sequence of vectors representing the isotropy groups associated with the boundary can be chosen to satisfy \eqref{eq:determinant_normalization}, using the fact that $T^2(v)=T^2(-v)$ and successively flipping the signs of the vectors appropriately.}

To prove that $\widehat{M}$ is homeomorphic to the closed upper half plane, first note that the orbit space $\widehat{M}$ is the quotient of a non-compact manifold by a compact Lie group and is therefore non-compact. We claim that $\widehat{M}$ is simply connected. To see this, take any loop $\widehat{\gamma}$ in $\widehat{M}$. Since $M$ has a fixed point, we can assume that the basepoint of $\widehat{\gamma}$ is the image of such a point. We can lift $\widehat{\gamma}$ to a curve $\gamma$ in $M$, and because the fiber of the basepoint of $\widehat{\gamma}$ consists of a single point, $\gamma$ must also be a loop. Using the assumption that $M$ is simply connected, we thus see that $\gamma$, and therefore $\widehat{\gamma}$ is null-homotopic.

Of the asymptotic geometries in Table \ref{tab:asymptotic_geometries}, all except ALH-splitting have only one end. Furthermore, $M$ cannot be ALH-splitting since by \cite[Theorem~3.7]{yZ29MGzxmmAh}, $M$ would then be homeomorphic to $\mathbb{R}\times T^3$, which contradicts the assumption that $M$ is simply connected. Thus, $M$ has only one end, and one can see from this that $\widehat{M}$ has only one end. By the classification of simply connected surfaces with boundary, $\widehat{M}$ is homeomorphic to the closed upper half plane.
\end{proof}
\begin{dfn}
The sequence $(v_0,\dots,v_n)$, where the vectors $v_i=(v_i^1,v_i^2)$ are as in Lemma \ref{lem:rod_structure_existence}, is called the \emph{rod structure} of $M$.
\end{dfn}
\begin{rem}
Some papers (e.g.,\ \cite{uguvpxtckkesrlwjgulj}) define the rod structure also to include the lengths $z_i-z_{i-1}$ of the boundary segments. Although these lengths are not well-defined in the context of Lemma \ref{lem:rod_structure_existence}, they are well-defined if one requires the homeomorphism from $\widehat{M}$ to the closed upper half plane to be given by a canonical set of coordinates, the \emph{Weyl--Papapetrou coordinates}.
\end{rem}
{\begin{lem}
Let $(M,g)$ be a toric gravitational instanton with rod structure $(v_0,\dots,v_n)$. Then the boundary at infinity of $M$ is the lens space\footnote{Here, we adopt the convention that $L(0,1)=S^2\times S^1$.} $L(p,q)$, where $p:=|\det\mqty(v_0&v_n)\/|$ and $q:=\operatorname{sgn}(\det\mqty(v_0&v_n))\cdot\det\mqty(v_1&v_n)$.
\label{lem:boundary_at_infinity_toric_instanton}
\end{lem}
\begin{proof}
Fix an identification of the orbit space $\widehat{M}$ with the closed upper half plane $\{z+i\rho\mid \rho\geq 0\}$ as in Lemma \ref{lem:rod_structure_existence}; such an identification can be chosen to be a diffeomorphism away from the corners. With $v_0,\dots,v_n$ and $z_1,\dots,z_n$ as in Lemma \ref{lem:rod_structure_existence}, set $R=\max(|z_1|,\dots,|z_n|)$, and consider the subset $\widehat{A}=\{z+i\rho\mid\rho\geq0,z^2+\rho^2>R^2\}\subseteq\widehat{M}$.

Without loss of generality, we may assume that $v_0=(0,1)$ and $v_1=(-1,0)$; otherwise, we can modify the $T^2$-action on $M$ by composing it with an appropriate isomorphism $T^2\to T^2$. Then $v_n=\pm(p,-q)$, so that the isotropy groups associated with the boundary segments of $\widehat{A}$ are $T^2(0,1)$ and $T^2(p,-q)$.

In the case where $v_0=v_n$, i.e.,\ $p=0$, the set $\widehat{A}$ can be realized as the orbit space of $(R,\infty)\times S^2\times S^1$. Here, the latter is endowed with the $T^2$-action given by \begin{equation}(e^{i\theta_1},e^{i\theta_2})\cdot(r,t,z_1,z_2)=(r,t,e^{i(\theta_2}z_1,e^{i\theta_1}z_2),\end{equation} where we are identifying $S^2$ as a subset of $\mathbb{R}\times\mathbb{C}$. By \cite[Theorem~1.1]{ypkrsoujnujrtxzxhykz}, this means that there exists a diffeomorphism $\overline{M\setminus K}\to[R+1,\infty)\times S^2\times S^1$ which maps $\partial K$ onto $\{R+1\}\times S^2\times S^1$, where $K$ is the preimage of the compact set $\widehat{K}=\{z+i\rho\mid\rho\geq0,z^2+\rho^2\leq(R+1)^2\}$ under the quotient map $M\to\widehat{M}$. Thus, $S^2\times S^1$ is the boundary at infinity of $M$.

Now consider the case where $v_0\neq v_n$, i.e.,\ $p\neq 0$. Define a $T^2$-action on $(R,\infty)\times L(p,q)$ by
\begin{equation}
(e^{i\theta_1},e^{i\theta_2})\cdot(r,z_1,z_2)=(r,e^{i\theta_1/p}z_1,e^{-i(\theta_2+q\theta_1/p)}z_2).
\end{equation}
Then points of the form $(r,z_1,0)$ have isotropy group $T^2(0,1)$, points of the form $(r,0,z_2)$ have isotropy group $T^2(p,-q)$, and all other points have trivial isotropy group, and in particular, the orbit space of $(R,\infty)\times L(p,q)$ can be realized as $\widehat{A}$. As in the previous case, the conclusion now follows directly from \cite[Theorem~1.1]{ypkrsoujnujrtxzxhykz}.
\end{proof}}
\begin{thm}
Let $(M,g)$ be a toric gravitational instanton. Then $(M,g)$ is either ALE with a cyclic group $\Gamma$, or ALF-$A_k$.\label{thm:asymptotic_geometries_toric_instanton}\end{thm}
\begin{proof}
{By Lemma \ref{lem:boundary_at_infinity_toric_instanton},} we can rule out any asymptotic geometry for which the boundary at infinity is not homotopy equivalent to $S^2\times S^1$ or to a lens space. Looking at Table \ref{tab:asymptotic_geometries}, this immediately rules out ALH-splitting since it is not even path-connected in that case. Using the fact that $\pi_1(L(p,q))$ is cyclic, we can also rule out ALH-non-splitting, along with any boundary of the form $S^3/\Gamma$ where $\Gamma$ is not cyclic, since $\pi_1(S^3/\Gamma)=\Gamma$. Thus, $M$ is neither ALF-$D_k$ with $k\neq 2$, ALE with non-cyclic group $\Gamma$, nor ALH-non-splitting.

It remains to rule out the cases ALF-$D_2$ and ALG. For ALF-$D_2$, note that $(S^2\times S^1)/\mathbb{Z}_2$ is homeomorphic to the quotient $(S^2\times\mathbb{R})/G$, where $G$ is the subgroup \begin{equation}\qty{\mqty(1&a\\0&b)\,\middle\vert\,b=\pm1,a\in\mathbb{Z}}\subseteq \operatorname{GL}(2,\mathbb{Z}),\end{equation} acting on $S^2\times\mathbb{R}$ by \begin{equation}\qty(\mqty(1&a\\0&b),(x_1,x_2,x_3,y))\mapsto(bx_1,bx_2,bx_3,by+a).\end{equation} The action of $G$ on $S^2\times\mathbb{R}$ is easily seen to be a covering space action (see \cite[p.~72]{MR1867354}), and since $S^2\times\mathbb{R}$ is simply connected, $G$ is the fundamental group of $(S^2\times S^1)/\mathbb{Z}_2$. But $G$ is not cyclic since it is infinite and contains the element $\mqty(1&0\\0&-1)$, whose order is $2$.

Finally, to rule out ALG, we will investigate the fundamental group of a fiber bundle $E\to S^1$ with fiber $T^2$. By \cite[Theorem~4.41, Proposition~4.48]{MR1867354}, there is an exact sequence of homotopy groups
\begin{equation}\begin{tikzcd}
\cdots \arrow[r] & \pi_2(S^1) \arrow[r] & \pi_1(T^2) \arrow[r] & \pi_1(E) \arrow[r] & \cdots.
\end{tikzcd}\end{equation} But the homotopy group $\pi_2(S^1)$ is trivial, since any continuous map $S^2\to S^1$ lifts to the universal cover $\mathbb{R}$, which is contractible. Since also $\pi_1(T^2)=\mathbb{Z}^2$, this means that there is an injective group homomorphism $\mathbb{Z}^2\to\pi_1(E)$, which cannot happen if $\pi_1(E)$ is cyclic.
\end{proof}
\begin{rem}
Consider $M=\mathbb{R}^2\times S^1\times S^1$, with a $T^2$-action where the first circle factor acts by rotating $\mathbb{R}^2$ around the origin, and the second circle factor acts by rotating one of the circle factors. Then the quotient space $\widehat{M}=M/T^2$ can be identified with $S^1\times[0,\infty)\cong\mathbb{R}^2\setminus B_1(0)$. Here, the inner circle represents points fixed by the subgroup corresponding to the first factor, while the exterior points correspond to points with a trivial isotropy group. This gives an analog of a rod structure for an ALH space, where there is one rod joined to itself at both ends and no turning points.
\end{rem}
\section{Examples}\label{sec:examples}
In this section, we consider some known examples of toric gravitational instantons. Table \ref{tab:examples_toric_instantons} lists their topology and topological invariants, along with their isometry groups, asymptotic geometries, and whether they are hyper-Kähler (see Definition \ref{dfn:hyper_kaehler}) or not. Their rod structures are shown in Figure \ref{fig:examples_rod_structures}. Further details about these instantons can be found in \cite{MR2659115,MR2831866,MR535152}.
\begin{figure}
\begin{minipage}[c][\textheight][c]{\textwidth}
\centering
\captionof{table}{Some known toric gravitational instantons.}
\begin{tabular}{llll}
\toprule
&Topology&Euler characteristic&Signature\\\midrule
Euclidean space and Taub--NUT&$\mathbb{R}^4$&1&0\\
Kerr and Schwarzschild&$\mathbb{R}^2\times S^2$&2&0\\
Taub-bolt&$\mathbb{C}P^2\setminus\{*\}$&2&1\\
Eguchi--Hanson&$T^*S^2$&2&1\\
Chen--Teo&$\mathbb{C}P^2\setminus S^1$&3&1\\
\bottomrule\\
\toprule
&Isometry group&Asymptotic geometry&Hyper-Kähler?\\\midrule
Euclidean space&$\mathbb{R}^4\rtimes\operatorname{O}(4)$&ALE with trivial group&Yes\\
Schwarzschild&$\operatorname{O}(2)\times\operatorname{O}(3)$&ALF-$A_{-1}$&No\\
Kerr&$\operatorname{O}(2)\times\operatorname{O}(2)$&ALF-$A_{-1}$&No\\
Taub--NUT&$(\operatorname{U}(1)\times\operatorname{SU}(2))/\mathbb{Z}_2$&ALF-$A_0$&Yes\\
Taub-bolt&$(\operatorname{U}(1)\times\operatorname{SU}(2))/\mathbb{Z}_2$&ALF-$A_0$&No\\
Eguchi--Hanson&$(\operatorname{U}(1)\times\operatorname{SU}(2))/\mathbb{Z}_2$&ALE with group $\mathbb{Z}_2$&Yes\\
Chen--Teo&$T^2$&ALF-$A_{-1}$&No\\
\bottomrule
\end{tabular}
\label{tab:examples_toric_instantons}
\vfill
\includegraphics[width=\textwidth]{R4_taub_nut.tikz}
\includegraphics[width=\textwidth]{schwarzschild_kerr.tikz}
\includegraphics[width=\textwidth]{taub_bolt.tikz}
\includegraphics[width=\textwidth]{eguchi_hanson.tikz}
\includegraphics[width=\textwidth]{chen_teo.tikz}
\captionof{figure}{Rod structures of (a) $\mathbb{R}^4$ and Taub--NUT; (b) Schwarzschild and Kerr; (c) Taub-bolt; (d) Eguchi--Hanson; (e) Chen--Teo.}
\label{fig:examples_rod_structures}
\end{minipage}     
\end{figure}
{\subsection{Euclidean Space}\label{sec:euclidean_space}
An elementary example is Euclidean space $\mathbb{R}^4$ with the flat metric. Identifying $\mathbb{R}^4$ with $\mathbb{C}^2$ and introducing coordinates $(r,\theta,\phi,\psi)$ by
\begin{equation}
(z_1,z_2)=\left(r\sin(\frac{\theta}{2})e^{i\phi},r\cos(\frac{\theta}{2})e^{i\psi}\right),
\end{equation}
where $r>0$ and $0<\theta<\pi$, the flat metric takes the form
\begin{equation}
g=\dd{r}^2+\frac{r^2}{4}\dd{\theta}^2+r^2\sin^2\left(\frac{\theta}{2}\right)\dd{\phi}^2+r^2\cos^2\left(\frac{\theta}{2}\right)\dd{\psi}^2.
\end{equation}
The Killing fields
\begin{equation}
\pdv{\psi}\qquad\text{and}\qquad\pdv{\phi}
\end{equation}
are $2\pi$-periodic, and thus generate an isometric $T^2$-action. Under this action, points with $\theta=0$ and $r>0$ have isotropy group $T^2(0,1)$, and points with $\theta=\pi$ and $r>0$ have isotropy group $T^2(1,0)$. The point where $r=0$ (i.e.,\ the origin) has isotropy group $T^2$, i.e.,\ is fixed under the whole action, and all other points have a trivial isotropy group.
\subsection{Riemannian Kerr and Schwarzschild Spaces}
Let $m>0$ and $a\in\mathbb{R}$, and introduce the shorthands $\Sigma=r^2-a^2\cos^2\theta$ and $\Delta=r^2-2mr-a^2$. The \emph{Riemannian Kerr metric},
\begin{equation}
g=\frac{\Sigma}{\Delta}\,dr^2+\Sigma\,d\theta^2+\frac{\Delta}{\Sigma}(d\tau-a\sin^2\theta\,d\phi)^2+\frac{\sin^2\theta}{\Sigma}((r^2-a^2)\,d\phi+a\,d\tau)^2,
\end{equation}
where $r>r_+:=m+\sqrt{m^2+a^2}$ and $0<\theta<\pi$, is a complete, Ricci-flat metric on $\mathbb{R}^2\times S^2$, provided the right identifications in the $(\tau,\phi)$-plane are made (see \cite[Section~3]{Nilsson_2024}).

The Killing fields
\begin{equation}
\frac{2mr_{+}}{\sqrt{m^2+a^2}}\pdv{\tau}-\frac{a}{\sqrt{m^2+a^2}}\pdv{\phi}\qquad\text{and}\qquad\pdv{\phi}
\end{equation}
are then $2\pi$-periodic, and correspond to rotation of the $\mathbb{R}^2$ factor around its origin, and rotation of the $S^2$ factor around its polar axis, respectively. Together, they generate an isometric $T^2$-action. Under this action, points in the set $(\mathbb{R}^2\setminus\{0\})\times\{N,S\}$ have isotropy group $T^2(0,1)$, where $N$ and $S$ are the north and south pole of $S^2$, respectively. Points in the set $\{0\}\times(S^2\setminus\{N,S\}$ have isotropy group $T^2(1,0)$, and the points $\{0\}\times\{N\}$ and $\{0\}\times\{S\}$ are fixed by all of $T^2$. All other points have a trivial isotropy group.

In the special case $a=0$, the metric is often referred to as the \emph{Riemannian Schwarzschild metric}, and in this case it reduces to
\begin{equation}
g=\left(1-\frac{2m}{r}\right)\dd{\tau}^2+\left(1-\frac{2m}{r}\right)^{-1}\dd{r}^2+r^2(\dd{\theta}^2+\sin^2\theta\dd{\phi}^2).
\end{equation}
\subsection{Riemannian Taub--NUT Space}
For $n>0$, the \emph{Riemannian Taub--NUT metric} is given by
\begin{equation}
g=\qty(1+\frac{2n}{r})(\dd{r^2}+r^2\dd{\theta^2}+r^2\sin^2\theta\dd{\phi^2})+\qty(1+\frac{2n}{r})^{-1}(\dd{\psi}+2n\cos\theta\dd{\phi})^2,
\end{equation}
where $r>0$ and $0<\theta<\pi$. Introducing new coordinates by
\begin{equation}
\begin{cases}
r&=\tilde{r}^2,\\
\psi&=2n(\tilde{\psi}+\tilde{\phi}),\\
\phi&=\tilde{\psi}-\tilde{\phi},
\end{cases}
\end{equation}
The metric takes the form
\begin{multline}
g=4(r+2n)\qty(\dd{\tilde{r}}^2+\frac{\tilde{r}^2}{4}\dd{\theta}^2+\tilde{r}^2\sin^2\qty(\frac{\theta}{2})\dd{\tilde{\phi}}^2+\tilde{r}^2\cos^2\qty(\frac{\theta}{2})\dd{\tilde{\psi}}^2)\\
-\frac{4(r+4n)}{r+2n}\qty(\tilde{r}^2\sin^2\qty(\frac{\theta}{2})\dd{\tilde{\phi}}+\tilde{r}^2\cos^2\qty(\frac{\theta}{2})\dd{\tilde{\psi}})^2.
\end{multline}
A comparison with the coordinate expressions in Section \ref{sec:euclidean_space} shows that $g$ can be made regular at the coordinate singularities $r=0$, $\theta=0$ and $\theta=\pi$ by giving $\psi$ and $\phi$ the periodicities $8n\pi$ and $2\pi$ respectively, independently. The metric $g$ is thus a complete Ricci-flat metric on $\mathbb{R}^4$. The $2\pi$-periodic Killing fields
\begin{equation}
\pdv{\tilde{\psi}}=2n\pdv{\psi}+\pdv{\phi}\qquad\text{and}\qquad\pdv{\tilde{\phi}}=2n\pdv{\psi}-\pdv{\phi}
\end{equation}
generate an isometric $T^2$-action, and when described in terms of the coordinates $r$ and $\theta$, the sets of points with specific isotropy groups are the same as those for Euclidean space with the $T^2$-action defined in Section \ref{sec:euclidean_space}.
\subsection{Taub-Bolt Space}
Let $n>0$. The \emph{Taub-bolt metric},
\begin{multline}
g=\frac{r^2-n^2}{(r-2n)(r-n/2)}\dd{r}^2+4n^2\frac{(r-2n)(r-n/2)}{r^2-n^2}(\dd{\psi}+\cos\theta\dd{\phi})^2\\
+(r^2-n^2)(\dd{\theta}^2+\sin^2\theta\dd{\phi}^2),
\end{multline}
is a complete, Ricci-flat metric on $\mathbb{C}P^2\setminus\{*\}$, given the right identifications in the $(\psi,\phi)$-plane (see \cite[Section~4]{Nilsson_2024}). Then, the Killing fields
\begin{equation}
\pdv{\psi}+\pdv{\phi}\qquad\text{and}\qquad\pdv{\psi}-\pdv{\phi}
\end{equation}
are both $2\pi$-periodic, and together they generate an isometric $T^2$-action. Under this action, points with $\theta=0$ and $r>2n$ have isotropy group $T^2(0,1)$, points with $\theta=\pi$ and $r>2n$ have isotropy group $T^2(1,0)$, and points with $0<\theta<\pi$ and $r=2n$ have isotropy group $T^2(1,1)$. The points $\theta=0$, $r=2n$ and $\theta=\pi$, $r=2n$, are fixed points of the $ T^2$ action, and all other points have a trivial isotropy group.
\subsection{Eguchi--Hanson Space}
Let $a>0$. The \emph{Eguchi--Hanson metric} is given by
\begin{equation}
g=\qty(1-\frac{a^4}{r^4})\frac{r^2}{4}(\dd{\psi}+\cos\theta\dd{\phi})^2+\qty(1-\frac{a^4}{r^4})^{-1}\dd{r^2}+\frac{r^2}{4}(\dd{\theta^2}+\sin^2\theta\dd{\phi^2}),
\end{equation}
where $r>a$ and $0<\theta<\pi$. Letting $r=(\tilde{r}^2+a^4)^{1/4}$ and $\psi=\tilde{\psi}+\phi$, we have
\begin{equation}
g=\frac{1}{4r^2}(\dd{\tilde{r}}^2+\tilde{r}^2\dd{\tilde{\psi}}^2)+\frac{r^2}{4}(\dd{\theta^2}+\sin^2\theta\dd{\phi^2})+O(\sin^4\theta)\dd{\phi}^2+O(\tilde{r}^2\sin^2\theta)\dd{\tilde{\psi}}\dd{\phi}.\label{eq:eguchi--hanson-regular-coordinates}
\end{equation}
This shows that $g$ can be made regular at the coordinate singularities $\theta=0$ and $r=a$ by independently giving each of the coordinates $\psi$ and $\phi$ periodicity $2\pi$.

Keeping the coordinate $\tilde{r}$, but introducing a coordinate $\hat{\psi}$ by $\psi=\hat{\psi}-\phi$, \eqref{eq:eguchi--hanson-regular-coordinates} still holds, but with $\hat{\psi}$ in place of $\tilde{\psi}$. Therefore, $g$ can also be made regular at $\theta=\pi$, with the same periodicities mentioned above. The metric $g$ is a complete Ricci-flat metric on $T^*S^2$, the cotangent bundle of the $2$-sphere (see \cite{MR0540896}).

The $2\pi$-periodic Killing fields $\pdv{\psi}$, and $\pdv{\phi}$ generate an isometric $T^2$-action, under which points with $\theta=0$ and $r>a$ have isotropy group $T^2(-1,1)$, points with $\theta=\pi$ and $r>a$ have isotropy group $T^2(1,1)$, and points with $r=a$ and $0<\theta<\pi$ have isotropy group $T^2(1,0)$. The points $\theta=0$, $r=a$ and $\theta=\pi$, $r=a$ are fixed under all of $T^2$, and all other points have a trivial isotropy group.
\subsection{Chen--Teo Space}\label{sec:chen--teo}
For parameters $\varkappa\in(0,\infty)$ and $\lambda\in(-1,1)$, let $\gamma=\sqrt{2-\lambda^2}$ and introduce the quantities \begin{equation}\kappa_E=\frac{2(\gamma-\lambda^2)}{\varkappa^2(1-\lambda^2)^2(1+\gamma)^2}\end{equation} and \begin{equation}\Omega_{1E}=-\frac{2(\gamma-\lambda^2)}{\varkappa^2(\gamma-\lambda)^2(1+\gamma)^2}.\end{equation} Defining the one-form \begin{multline}\Omega=\frac{\varkappa^2(1-\lambda^2)(1+\gamma)(x+\lambda)(y+\lambda)}{2(2+\gamma)(x-y)F(x,y)}\bigg(\\2(x+1)(y-1)(\lambda(3+\gamma)(2(x+\lambda)(y+\lambda)+(1+\gamma)(\gamma x-\gamma y-2))+3(1-\lambda^2)(x+y))\\+(1+\lambda)^3(\gamma-\lambda+2)^2(x^2-1)-(1-\lambda)^3(\gamma+\lambda+2)^2(y^2-1)\bigg)\dd{\phi},\end{multline} along with the functions $G(x)=(1-x^2)(x+\lambda)$, \begin{multline}H(x,y)=\frac{(\lambda+\gamma)(x+\lambda y)+(\lambda-\gamma)(y-\lambda x)+2xy+2\gamma\lambda^2}{4(2+\gamma)}\bigg(\\(2+\gamma)(x+\lambda)(y+\lambda)(\lambda(1+\gamma)(x+y)+(\lambda^2-\gamma)(x-y)+4+2\gamma\lambda^2-2xy)\\+(1-\lambda^2)((\gamma+\lambda+2)(x+\lambda)(x+\lambda\gamma)+(\gamma-\lambda+2)(y+\lambda)(y+\lambda\gamma))\bigg),\end{multline} and \begin{equation}F(x,y)=(x+\lambda)(y+\lambda)((1+\lambda\gamma)^2(xy+\lambda x+\lambda y+1)-2\lambda\gamma(1-\lambda)(x-1)(y-1)-(x^2-1)(y^2-1)),\end{equation} we define a metric by \begin{equation}g=\frac{F(x,y)}{H(x,y)}(\dd{\psi}+\Omega)^2-\frac{\varkappa^4 H(x,y)}{(x-y)^3}\qty(\frac{\dd{x^2}}{G(x)}-\frac{\dd{y^2}}{G(y)}+\frac{4G(x)G(y)}{(x-y)F(x,y)}\dd{\phi^2}),\end{equation} where $x<-1$ and $y>1$. The metric $g$ will be referred to as the \emph{Chen--Teo metric}.

With the identifications $(\psi,\phi)\sim(\psi+2\pi/\kappa_E,\phi+\Omega_{1E}\phi/\kappa_E)\sim(\psi,\phi+2\pi)$, the Killing fields
\begin{equation}
\frac{1}{\kappa_E}\pdv{\psi}+\frac{\Omega_{1E}}{\kappa_E}\pdv{\phi}\qquad\text{and}\qquad\pdv{\phi}
\end{equation}
become $2\pi$-periodic, and generate an isometric $T^2$-action. Introducing appropriate radial and angular coordinates and performing a Taylor expansion, reasoning as in the previous sections, one can show that $g$ can be made regular at points where $x\in\{-\infty,-1\}$ or $y\in\{1,\infty\}$, except at points where $(x,y)=(-\infty,\infty)$. This makes it into a complete Ricci-flat metric on $\mathbb{C}P^2\setminus S^1$.

Under the $T^2$-action, points where $x=-\infty$ and $y>1$ have isotropy group $T^2(0,1)$, and the same is true for points where $y=\infty$ and $x<-1$. Points where $y=1$ and $-\infty<x<-1$ have isotropy group $T^2(1,0)$, while points where $x=-1$ and $1<y<\infty$ have isotropy group $T^2(1,1)$. Points with $(x,y)\in\{(-\infty,1),(-1,1),(-1,\infty)\}$ are fixed under the whole action.}
\section{Conditions on the Topology}\label{sec:topological_conditions}
In this section, we prove Theorem \ref{thm:rod_structure_intersection_form} by explicitly reconstructing $M$ as a smooth manifold with a $T^2$-action directly from its rod structure. We then calculate the intersection form for the reconstructed model.
\begin{proof}[Proof of Theorem \ref{thm:rod_structure_intersection_form}]
Let $M_1,\dots,M_n$ denote distinct copies of $\mathbb{R}^4\cong\mathbb{C}^2$, and for each $i$, define a $T^2$-action on $M_i$ by \begin{equation}(e^{i\theta_1},e^{i\theta_2})\cdot(z_1,z_2)=(e^{i(a_{11}\theta_1+a_{12}\theta_2)}z_1,e^{i(a_{21}\theta_1+a_{22}\theta_2)}z_2),\end{equation} where \begin{equation}\mqty(a_{11}&a_{12}\\a_{21}&a_{22})=\mqty(v_{i-1}&v_i)^{-1}.\end{equation} Now let $M_i^+=\{(z_1,z_2)\in M_i\mid |z_1|>|z_2|^2+1\}$ and $M_i^-=\{(z_1,z_2)\in M_i\mid |z_2|>|z_1|^2+1\}$. Since \begin{equation}\mqty(v_i&v_{i+1})^{-1}v_i=\mqty(1\\0),\end{equation} and since $\det(\mqty(v_i&v_{i+1})^{-1}\mqty(v_{i-1}&v_i))=1$, we can write \begin{equation}\mqty(v_i&v_{i+1})^{-1}\mqty(v_{i-1}&v_i)=\mqty(-d_i&1\\-1&0)\end{equation} for some $d_i\in\mathbb{Z}$, and since $v_{i-1}=-d_iv_i-v_{i+1}$, we have $\det\mqty(v_{i-1}&v_{i+1})=-d_i\det\mqty(v_i&v_{i-1})=-d_i$. We have diffeomorphisms $F_i:M_i^+\to M_{i+1}^-$ given by \begin{equation}F_i(z_1,z_2)=\qty(\qty(\frac{z_1}{|z_1|})^{-d_i}z_2,\qty(\frac{z_1}{|z_1|})^{-1}\qty(|z_2|^2+1+\frac{1}{|z_1|-|z_2|^2-1})),\end{equation} and gluing the sets $M_i$ together along these diffeomorphisms, we obtain a manifold $M'$. It can be checked that the diffeomorphisms are equivariant with respect to the $T^2$-actions defined on the sets $M_i$ and that they are orientation-preserving, where each $M_i$ is given the standard orientation of $\mathbb{R}^4$. This makes $M'$ into an oriented manifold with a $T^2$-action. By construction, we can map the orbit space $M'/T^2$ diffeomorphically onto the orbit space $M/T^2$ in such a way that the isotropy groups match, and by \cite[Theorem~1.1]{ypkrsoujnujrtxzxhykz}, this implies that $M$ and $M'$ are equivariantly diffeomorphic. For simplicity, we therefore identify $M$ with $M'$.

{For $1\leq i\leq n-1$, define the set \begin{equation}R_i=\{(z_1,z_2)\in M_i\mid z_2=0\}\cup\{(z_1,z_2)\in M_i\mid z_1=0\}.\end{equation} This is an embedded $2$-sphere in $M$, and in particular it is a closed orientable submanifold of $M$. We give $R_i$ the orientation for which the standard coordinate vector fields $(\pdv*{x_1},\pdv*{x_2})$ for $M_i$ form a positively oriented frame for $R_i$. Then the fundamental classes $[R_1],\dots,[R_{n-1}]$ are elements of $H_2(M)$, and we claim that they form a basis.

To this end, assume that $n>1$ (otherwise, there is nothing to prove). For $1\leq i\leq n-1$, we have the Mayer--Vietoris sequence
\begin{equation}
H_2(M_i)\oplus H_2(M_{i+1})\to H_2(M_i\cup M_{i+1})\to H_1(M_i\cap M_{i+1})\to H_1(M_i)\oplus H_1(M_{i+1}).
\end{equation}
Since $M_i\cong M_{i+1}\cong\mathbb{R}^4$, the first and last terms vanish, which implies that the middle map is an isomorphism. We can decompose $R_i$ as the union of two closed disks $D_i\subseteq M_i$ and $D_{i+1}\subseteq M_{i+1}$, joined along their common boundary circle $D_i\cap D_{i+1}\subseteq M_i\cap M_{i+1}$. The middle map can be described explicitly (see, e.g.,\ \cite[p.~150]{MR1867354}), and it maps $[R_i]$ to $[D_i\cap D_{i+1}]$, where $D_i\cap D_{i+1}$ has been oriented appropriately. But the intersection $M_i\cap M_{i+1}$ deformation retracts to $D_i\cap D_{i+1}$, which shows that $\{[D_i\cap D_{i+1}]\}$ is a basis for $H_1(M_i\cap M_{i+1})$. Hence, it follows that $\{[R_i]\}$ is a basis for $H_2(M_i\cup M_{i+1})$.

Now assume that $2\leq i\leq n-1$ and put $A=M_1\cup\cdots\cup M_i$ and $B=M_i\cup M_{i+1}$, and consider the Mayer--Vietoris sequence
\begin{equation}
H_2(A\cap B)\to H_2(A)\oplus H_2(B)\to H_2(A\cup B)\to H_1(A\cap B).
\end{equation}
The first and last terms vanish because $A\cap B=M_i\cong\mathbb{R}^4$, which shows that the middle map is an isomorphism. From the previous paragraph we know that $\{[R_1]\}$ is a basis for $H_2(M_1\cup M_2)$, and by inducting on $i$ we can conclude that $\{[R_1],\dots,[R_{n-1}]\}$ is a basis for $H_2(M)$.

What remains is to compute the intersection form with respect to this basis. For $|i-j|>1$, the submanifolds $R_i$ and $R_j$ are disjoint, and thus the intersection product $[R_i]\cdot [R_j]$ vanishes. Two adjacent submanifolds $R_i$ and $R_{i+1}$ intersect in exactly one point, and they do so transversely, and from this, it follows that $[R_i]\cdot[R_{i+1}]=\pm1$. Flipping orientations of the submanifolds $R_i$ appropriately, we ensure that $[R_i]\cdot[R_{i+1}]=1$ for all $i$, noting that flipping orientation does not affect the self-intersection numbers $[R_i]\cdot[R_i]$

To calculate the self-intersection numbers $[R_i]\cdot[R_i]$, we have to perturb $R_i$ into another submanifold representative of the same homology class, which intersects $R_i$ transversely. We first consider the case where $d_i\geq0$. In this case, we take a smooth function $f:\mathbb{C}\to\mathbb{C}$ such that $f(z)=z^{d_i}-3^{-d_i}$ when $|z|\leq 2/3$, and $f(z)=(z/|z|)^{d_i}$ when $|z|\geq1$, and such that $f(z)\neq0$ when $2/3\leq|z|\leq 1$. Now define a submanifold by \begin{equation}R_i'=\{(z_1,z_2)\in M_i\mid z_2=f(z_1)\}\cup\{(z_1,z_2)\in M_{i+1}\mid z_1=1\}.\end{equation} The inclusion map $R_i'\to M$ is homotopic to an embedding whose image is $R_i$. This can be seen by considering the homotopy $R_i'\times[0,1]\to M$ which in $M_i$ is given by $((z_1,z_2),t)\mapsto(z_1,(1-t)z_2)$, and in $M_{i+1}$ is given by $((z_1,z_2),t)\mapsto((1-t)z_1,z_2)$. Endowing $R_i'$ with the appropriate orientation, this shows that $R_i'$ is homologous to $R_i$.}

Since $f$ has exactly $d_i$ roots, and since it is holomorphic with non-vanishing derivative near these roots, it follows that $R_i$ and $R_i'$ intersect transversely in $d_i$ points and that they intersect positively at each such point (with respect to the orientation of $M$). In other words, $[R_i]\cdot [R_i]=d_i$. The case where $d_i<0$ is exactly the same, except that we let $f(z)=\overline{z}^{-d_i}-3^{d_i}$ for $|z|\leq 2/3$. In that case, the function $f$ is antiholomorphic (instead of holomorphic) in that region so that $R_i$ and $R_i'$ intersect negatively at each root.
\end{proof}

\subsection{Hitchin--Thorpe Inequality}\label{sec:hitchin_thorpe_inequality}
The following definition is needed to state the results in this section.
\begin{dfn}
A Riemannian manifold $(M,g)$ is said to be \emph{hyper-Kähler} if it admits three almost complex structures $I,J$ and $K$ such that
\begin{itemize}
\item $(M,g)$ is a Kähler manifold with respect to $I,J$ and $K$ separately.
\item $I^2=J^2=K^2=-1$.
\end{itemize}\label{dfn:hyper_kaehler}
\end{dfn}
In its original form, the Hitchin--Thorpe inequality is a statement about compact Einstein $4$-manifolds. For such manifolds $M$, the inequality states that \begin{equation}2\chi(M)\geq3|\tau(M)|,\label{eq:hitchin_thorpe_compact}\end{equation} and it further states that equality occurs if and only if the universal cover of $M$ is hyper-Kähler. As a special case, this inequality applies when $M$ is Ricci-flat. In general, the inequality does not hold without modification in the non-compact case. For Ricci-flat ALE or ALF manifolds, there are, however, variants of the inequality that involve extra boundary terms. In the case of ALE manifolds, for instance, we have the following variant of the inequality, whose proof can be found in \cite[Theorem~4.2]{MR1145266}.
\begin{thm}[Hitchin--Thorpe Inequality for Ricci-Flat ALE Manifolds]
Let $(M,g)$ be an oriented Ricci-flat ALE manifold with group $\Gamma$. Then \begin{equation}2\qty(\chi(M)-\frac{1}{|\Gamma|})\geq 3|\tau(M)+\eta_S(S^3/\Gamma)|.\label{eq:hitchin_thorpe_ALE}\end{equation} Equality holds if and only if the universal cover of $M$ is hyper-Kähler.\label{thm:hitchin_thorpe_ALE}
\end{thm}
The number $\eta_S(S^3/\Gamma)$ occurring in Theorem \ref{thm:hitchin_thorpe_ALE} is the so-called \emph{eta-invariant of the signature operator of the space form $S^3/\Gamma$}, which is a spectral invariant of this space form. Although we will not describe the eta-invariant in detail, we mention that the eta-invariant is an oriented isometry invariant of the space form $S^3/\Gamma$, and that reversing the orientation of this space form flips the sign of the eta-invariant.

{The orientation of $S^3/\Gamma$ that defines the eta-invariant is the boundary orientation on $S^3/\Gamma\cong\{R\}\times S^3/\Gamma$ when viewed as the boundary of $\{x\in\mathbb{R}^4/\Gamma\mid|x|\leq R\}$, where the latter is oriented consistently with the asymptotic diffeomorphism, giving rise to an induced orientation on $L(p,q)$. Replacing the group $\Gamma$ by its conjugate by an orientation-reversing element of $\mathrm{O}(n)$ if necessary, we may assume that the orientation is induced by the standard orientation on $\mathbb{R}^4$.

Now let $M$ be a toric ALE (with group $\Gamma$) instanton, with rod structure $(v_0,\dots,v_n)$, where we assume that $\det\mqty(v_{i-1}&v_i)=1$. Then Theorem \ref{thm:hitchin_thorpe_ALE} applies, and since $M$ is simply connected, the statement about its universal cover applies directly to $M$. From Lemma \ref{lem:boundary_at_infinity_toric_instanton}, it follows that the boundary at infinity of $M$ is $L(p,q)$, where $p=|\det\mqty(v_0&v_n)\/|$ and $q=\operatorname{sgn}(\det\mqty(v_0&v_n))\cdot\det\mqty(v_1&v_n)$, and Lemma \ref{lem:boundary_at_infinity_well_defined} implies that $S^3/\Gamma$ and $L(p,q)$ are $h$-cobordant. In particular, this implies that $\Gamma$ is a cyclic group of order $p$, and by the classification of spherical $3$-manifolds (see e.g.,\ \cite{MR0426001}), $S^3/\Gamma=L(p,q')$ for some $q'$.

Let $M$ be equipped with the orientation defined in the proof of Theorem \ref{thm:rod_structure_intersection_form}. By following the proofs of Lemmas \ref{lem:boundary_at_infinity_toric_instanton} and \ref{lem:boundary_at_infinity_toric_instanton} carefully,
one verifies that the $h$-cobordism is an \emph{oriented} $h$-cobordism, by which we mean that the cobordism $W$ admits an orientation for which the induced boundary orientation on $\partial W$ is that of $(-L(p,q))\sqcup L(p,q')$. By \cite[Theorem~1]{doig2015combinatorial}, this implies that $L(p,q)$ and $L(p,q')$ are orientation-preservingly homeomorphic. The classification of lens spaces then implies that $L(p,q)$ and $L(p,q')$ are orientation-preservingly isometric so that $\eta_S(L(p,q))=\eta_S(L(p,q'))$.} We now cite the following formula for the eta-invariants of lens spaces, whose proof can be found in \cite[Theorem~4]{MR926246}.

\begin{thm}
For $p,q\in\mathbb{Z}$ with $0\leq q<p$ and $\operatorname{gcd}(p,q)=1$, the eta-invariant of the signature operator of $L(p,q)$ is
\begin{equation}
\eta_S(L(p,q))=\frac{1}{3p}(p-1)({2pq-3p}-q+3)-\frac{2}{p}\sum_{k=1}^{q-1}\left\lfloor\frac{kp}{q}\right\rfloor^2.\label{eq:eta_invariant_lens_space}
\end{equation}\label{thm:eta_invariant_lens_space}
\end{thm}

{By the discussion in the previous paragraph, the eta-invariant which occurs in the inequality \eqref{eq:hitchin_thorpe_ALE} for the toric ALE instanton $M$ is given by the expression on the right in \eqref{eq:eta_invariant_lens_space}, where $p$ and $q$ are given as in the previous paragraph, except that $q$ is reduced\footnote{Note that reduction of $q$ modulo $p$ does not change the lens space $L(p,q)$.} modulo $p$ in order to satisfy $0\leq q<p$.

It is also well known (and can be seen directly by considering the construction in the proof of Theorem \ref{thm:rod_structure_intersection_form}) that the Euler characteristic is given by $\chi(M)=n$, and taken together, this shows that we can express all of the quantities occurring in Theorem \ref{thm:hitchin_thorpe_ALE} directly in terms of the rod structure. Theorem \ref{thm:hitchin_thorpe_ALE} can thus be interpreted as a necessary condition that rod structures of toric ALE instantons have to satisfy.}

As mentioned, there are also variants of the Hitchin--Thorpe inequality for ALF manifolds. In the case of ALF-$A_k$, we have the following result, whose proof can be found in \cite[Theorem~6.12]{4797854091}.
\begin{thm}[Hitchin--Thorpe Inequality for Ricci-Flat ALF-$A_k$ Manifolds]
Let $(M,g)$ be an oriented Ricci-flat manifold which is ALF-$A_k$ for some integer $k$. Then \begin{equation}2\chi(M)\geq 3\qty|\tau(M)-\frac{e}{3}+\operatorname{sgn}(e)|,\label{eq:hitchin_thorpe_ALF}\end{equation} where $e=-k-1$. Equality holds if and only if the universal cover of $M$ is hyper-Kähler.\label{thm:hitchin_thorpe_ALF}
\end{thm}
{Here, the asymptotic diffeomorphism $\overline{M\setminus K}\to(R,\infty)\times L(|e|,\operatorname{sgn}e)$ to be orientation-preserving. For $p\neq 1,2$, there is no orientation-preserving diffeomorphism $L(p,1)\to L(p,-1)$, which shows that purely topological arguments uniquely determine the sign of $e$ in this case. The cases $|e|=1$ and $|e|=2$ are slightly more subtle and require further consideration of how the asymptotic diffeomorphism relates to the metrics of the spaces, but in this paper, we will not consider these aspects.}

\subsection{Rod structures With Three Turning Points}\label{sec:three_turning_points}
Consider a rod structure $(v_0,v_1,v_2,v_3)$ with three turning points, satisfying $v_0=(0,1)$ and $v_1=(-1,0)$; any such rod structure can be written as in Figure \ref{fig:three_turning_points_rod_structure}, for some $a,b\in\mathbb{Z}$. For a toric gravitational instanton $(M,g)$ with this rod structure, we have $\chi(M)=3$, and the signature $\tau(M)$ is just the difference between the number of positive and negative roots, respectively, of the polynomial $\lambda^2-(a+b)\lambda+ab-1$. Assume now, in addition, that $(M,g)$ is ALE\@. As we shall see, this additional assumption restricts the possible pairs $(a,b)$. {An obvious restriction is, of course, that $p:=\det\mqty(v_0&v_3\/)=|1-ab|$ is non-zero, for otherwise, the boundary at infinity would be $S^2\times S^1$, which is incompatible with ALE geometry. Moreover, we must have $p>1$, since if $p=1$ then $M$ is AE, so the positive mass theorem (see \cite{MR0994021}) implies that $M$ is homeomorphic to $\mathbb{R}^4$, which contradicts the fact that $\chi(M)=3$.

Now let $q$ be the unique integer in the range $0\leq q<p$ which satisfies $q\equiv b\operatorname{sgn}(1-ab)\pmod{p}$. By Theorem \ref{thm:hitchin_thorpe_ALE}, we have
\begin{equation}
2\left(3-\frac{1}{p}\right)\geq 3|\tau(M)+\eta_S(L(p,q))|,\label{eq:three_turning_points_hitchin_thorpe_ALE}
\end{equation}
where $\eta_S(L(p,q))$ is given by the right hand side of \eqref{eq:eta_invariant_lens_space}.

The requirement that $p>1$, along with the inequality \eqref{eq:three_turning_points_hitchin_thorpe_ALE}, can be viewed as restrictions on the values of $a$ and $b$, and for any specific values of $a$ and $b$, it is straightforward to check whether or not they. In other words, if for some $a,b\in\mathbb{Z}$ one of these requirements does \emph{not} hold, there cannot exist any toric ALE instanton with the corresponding rod structure. By systematically checking all pairs $(a,b)$ with $|a|,|b|\leq 17$, one arrives at Figure \ref{fig:three_turning_points_ALE_possibilities}. The pairs $(a,b)=\pm(2,2)$ (marked with red) are the pairs for which exact equality holds in \eqref{eq:three_turning_points_hitchin_thorpe_ALE}. They correspond to the unique hyper-Kähler instanton on the minimal resolution of $\mathbb{C}^2/\mathbb{Z}_3$ (see \cite{MR0992335}), equipped with two distinct $T^2$-actions.

Consider, now, the same question for ALF: Given $a,b\in\mathbb{Z}$, does there exist a toric ALF instanton $(M,g)$ with the rod structure in Figure \ref{fig:three_turning_points_possibilities}? Such a manifold $M$ is necessarily of the type ALF-$A_k$ With the same definitions of $p$ and $q$ as before, and with $e=-k-1$, we have $|e|=p$, and either $q=1$ or $q=p-1$ holds. Furthermore, Theorem \ref{thm:hitchin_thorpe_ALF} implies that
\begin{equation}
6\geq 3\left|\tau(M)-\frac{e}{3}+\operatorname{sgn}(e)\right|.\label{eq:three_turning_points_hitchin_thorpe_ALF}
\end{equation}
When $p\neq1,2$, the sign of $e$ is uniquely determined, and is given by
\begin{equation}
e=
\begin{cases}
-p,&q=1,\\
p,&q=p-1,
\end{cases}
\end{equation}
while in the cases $p=1$ and $p=2$, the sign $e$ is ambiguous. Nonetheless, in the latter cases, \eqref{eq:three_turning_points_hitchin_thorpe_ALF} must hold with either $e=p$ \emph{or} $e=-p$.

If equality holds in \eqref{eq:three_turning_points_hitchin_thorpe_ALF}, then $M$ is hyper-Kähler, and by the classification of hyper-Kähler ALF instantons, $M$ is a triple-Taub--NUT instanton (see \cite{MR2818855}). The triple-Taub--NUT instanton has Euler characteristic equal to $p$, which implies that $p=3$, and it has an even intersection form, i.e.\ $d_1=a$ and $d_2=b$ are even.

Again, the preceding statements restrict the possible values of $a$ and $b$. Using the same method to compute the signature $\tau(M)$ as before, systematically checking all pairs $(a,b)$ with $|a|,|b|\leq 17$, one arrives at Figure \ref{fig:three_turning_points_ALF_possibilities}. As for the
ALE case, the pairs $(a,b)=\pm(2,2)$ (marked with red) are the pairs for which exact equality holds in \eqref{eq:three_turning_points_hitchin_thorpe_ALF}. They correspond to two
distinct $T^{2}$-actions on the triple-Taub--NUT instanton. Furthermore, the pairs given by $(a,b)=\pm(1,1)$ (marked with blue) correspond to two distinct $T^{2}$-actions on the Chen--Teo instanton.

\begin{proof}[Proof of Theorem \ref{thm:three_turning_points_possibilities}]
For the ALF statement, \eqref{eq:three_turning_points_hitchin_thorpe_ALF} implies that
\begin{equation}
6\geq |e|-3|\tau(M)+\operatorname{sgn}(e)|\geq |e|-9,
\end{equation}
so that
\begin{equation}
|ab|\leq 1+|1-ab|=1+|e|\leq 16.
\end{equation}
Thus, either $|a|,|b|\leq 16$, in which case $(a,b)$ is one of the points in Figure \ref{fig:three_turning_points_ALF_possibilities}, or either $a$ or $b$ vanishes.

We now prove the ALE statement. Since the transformation $(a,b)\mapsto(-a,-b)$ corresponds to $(p,q)\mapsto(p,p-q)$, and since the latter corresponds to an orientation reversal of the lens space $L(p,q)$, it follows that $(a,b)\mapsto(-a,-b)$ amounts to a sign reversal of $\eta_S(L(p,q))$. The transformation $(a,b)\mapsto(-a,-b)$ also leads to a sign reversal of $\tau(M)$, and thus preserves the inequality \eqref{eq:three_turning_points_hitchin_thorpe_ALE}. Since the conclusion of the theorem is also preserved under this transformation, we may assume without loss of generality that $a\leq 0$. Moreover, since $p>1$, we must in fact have $a<0$ and $b\neq 0$. 

First assume that $b>0$; then $p=1-ab$, and since $0<b<1-ab$, we have $q=b$. For every integer $k$ with $1\leq k\leq q-1$, we have
\begin{equation}
\left\lfloor\frac{kp}{q}\right\rfloor=\left\lfloor\frac{k}{b}-ka\right\rfloor=-ka,
\end{equation}
and by \eqref{eq:eta_invariant_lens_space}, this implies that
\begin{equation}
\eta_S(L(p,q))=-\frac{ab(a+b)}{3(1-ab)}.
\end{equation}
Since the constant coefficient of the polynomial $\lambda^2-(a+b)\lambda+ab-1$ is negative, the roots have different signs, which means that $\tau(M)=0$, and \eqref{eq:three_turning_points_hitchin_thorpe_ALE} now implies that
\begin{equation}
6-\frac{2}{1-ab}\geq\frac{|ab(a+b)|}{1-ab},
\end{equation}
or equivalently,
\begin{equation}
(|a+b|-6)|ab|\leq 4.
\end{equation}
This means that either $|a+b|\leq 6$, in which case we are done, or $|ab|\leq 4$. In the latter case,
\begin{equation}
(a+b)^2=a^2+b^2+2ab\leq\frac{16}{b^2}+\frac{16}{a^2}+8\leq40,
\end{equation}
which also implies that $|a+b|\leq 6$.

Consider now the case $a,b<-1$, for which $p=ab-1$ and $q=-b$. For $1\leq k\leq q-1$ we have
\begin{equation}
\left\lfloor\frac{kp}{q}\right\rfloor=\left\lfloor\frac{k}{b}-ka\right\rfloor=-ka-1,
\end{equation}
so \eqref{eq:eta_invariant_lens_space} implies that
\begin{equation}
\eta_S(L(p,q))=\frac{a^2b+ab^2}{3(ab-1)}+2.
\end{equation}
One readily verifies that $\tau(M)=-2$, and thus \eqref{eq:three_turning_points_hitchin_thorpe_ALE} implies that
\begin{equation}
6-\frac{2}{ab-1}\geq\frac{|a^2b+ab^2|}{ab-1},
\end{equation}
or equivalently,
\begin{equation}
6-\frac{8}{ab}\geq|a+b|.
\label{eq:unnamed_equation_1}
\end{equation}
In particular, $|a+b|<6$, so that $(a,b)$ is either $(-2,-2)$, $(-2,-3)$ or $(-3,-2)$. As the latter two do not satisfy \eqref{eq:unnamed_equation_1}, we must have $(a,b)=(-2,-2)$.

In the case $a=-1$, $b<-1$, we have $\tau=-2$, $p=-b-1$ and $q=1$, and
\begin{equation}
\eta_S(L(p,q))=\frac{(b+2)(b+3)}{3(b+1)},
\end{equation}
so \eqref{eq:three_turning_points_hitchin_thorpe_ALE} implies that
\begin{equation}
-6b-8\geq|(1-b)b|,
\end{equation}
which is a contradiction. The case $a<-1$, $b=-1$ is almost the same, the only difference being that $a$ occurs in place of $b$ in the expressions for $p$ and $\eta_S(L(p,q))$.
\end{proof}}
\begin{rem}
One can also look at rod structure with four turning points. As in the case with three turning points, we can assume that $v_0=(0,1),v_1=(-1,0)$. Such a rod structure will have the form shown in Figure \ref{fig:four_turning_points}, with three integer parameters $a,b,c$. Restricting attention to the AF case, we are left with two families of rod structures, shown in Figure \ref{fig:four_turning_points_AF_families}, along with four exceptional rod structures, shown in Figure \ref{fig:four_turning_points_exceptional_AF_rod_structures}. For AF rod structures with four turning points, it then turns out that the Hitchin--Thorpe inequality is always satisfied with strict inequality. Thus, we cannot rule out any of these rod structures.

Since the inequalities are strict, we can at least conclude that a simply connected AF toric gravitational instanton with four turning points cannot be hyper-Kähler. However, this is already known: by the classification of ALF-$A_k$ hyper-Kähler gravitational instantons, any AF hyper-Kähler gravitational instanton must be a product of $\mathbb{R}^3$ with a circle, which is not simply connected.

In other words, the Hitchin--Thorpe inequality gives no information about AF rod structures with four turning points.
\end{rem}
\begin{figure}
\centering
\includegraphics[width=\textwidth]{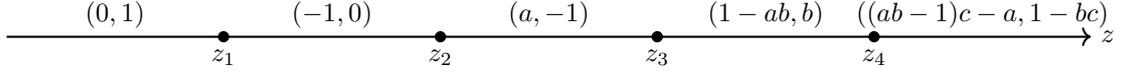}
\caption{A general rod structure with four turning points.}\label{fig:four_turning_points}
\end{figure}
\begin{figure}
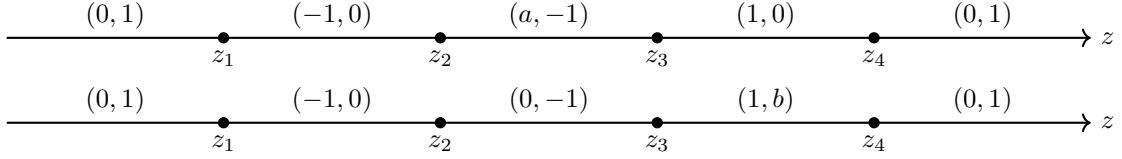

\centering
\includegraphics[width=\textwidth]{four_turning_points_AF_rod_structure_1.tikz}
\includegraphics[width=\textwidth]{four_turning_points_AF_rod_structure_2.tikz}
\caption{Two families of AF rod structures with four turning points.}\label{fig:four_turning_points_AF_families}
\end{figure}
\begin{figure}
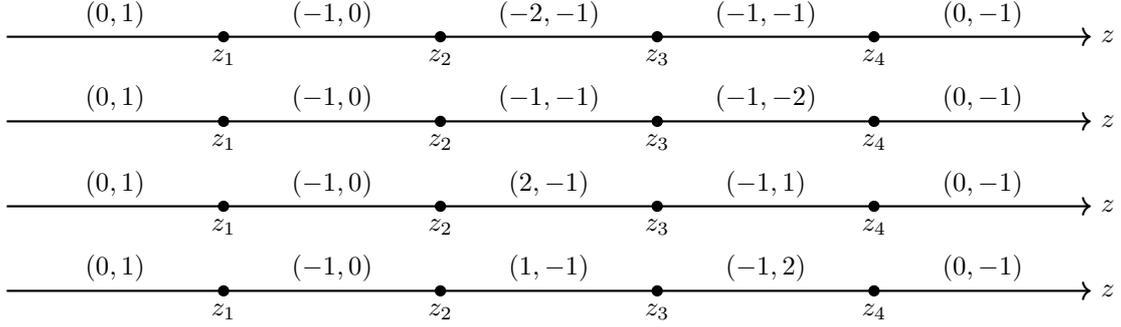

\centering
\includegraphics[width=\textwidth]{four_turning_points_AF_exceptional_1.tikz}
\includegraphics[width=\textwidth]{four_turning_points_AF_exceptional_2.tikz}
\includegraphics[width=\textwidth]{four_turning_points_AF_exceptional_3.tikz}
\includegraphics[width=\textwidth]{four_turning_points_AF_exceptional_4.tikz}
\caption{Four exceptional AF rod structures with four turning points.}\label{fig:four_turning_points_exceptional_AF_rod_structures}
\end{figure}
\bibliographystyle{plain}
\bibliography{main}
\end{document}